\documentclass{amsart}
\usepackage{mypreamble}

\newtheorem{thm}[subsubsection]{Theorem}
\newtheorem{lem}[subsubsection]{Lemma}
\newtheorem{prop}[subsubsection]{Proposition}
\newtheorem{cor}[subsubsection]{Corollary}

\theoremstyle{remark}
\newtheorem{rem}[subsubsection]{Remark}

\theoremstyle{definition}

\numberwithin{equation}{section}

\nc\C{\eC}
\nc\Cm{\C_-}
\nc\Cp{\C_+}
\nc\CC{\sC}
\nc\CCp{\sC_+}
\nc\CCm{\sC_-}
\nop\Inv{Inv}

\title{Radon inversion formulas over local fields}
\author{Jonathan Wang}

\begin{document} 
\maketitle

\begin{abstract}
Let $F$ be a local field and $n\ge 2$ an integer. 
We study the Radon transform as an operator 
$M : \Cp \to \Cm$ from the space of smooth $K$-finite functions
on $F^n \setminus \{0\}$ with bounded support to the space
of smooth $K$-finite functions on $F^n \setminus \{0\}$ supported away 
from a neighborhood of $0$.  
These spaces naturally arise in the theory of automorphic forms. 
We prove that $M$ is an isomorphism and provide formulas for $M^{-1}$. 
In the real case, we show that when $K$-finiteness
is dropped from the definitions, the analog of $M$ is not surjective. 

\end{abstract}

\section{Introduction}

\subsection{Some notation}
\subsubsection{}
Let $F$ be a local field  
(i.e., $F$ is either non-Archimedean or $\bbR$ or $\bbC$).
Let $G$ denote the topological group $\GL_n(F)$ for an integer $n\ge 2$. 

Let $K$ be the standard maximal 
compact subgroup of $G$ (i.e., if $F$ is non-Archimedean then $K = \GL_n(\eO)$, where $\eO\subset F$
is the ring of integers, if $F = \bbR$ then $K = O(n)$, and if $F = \bbC$ then $K = U(n)$). 

\subsubsection{} 
We fix a field $E$ of characteristic $0$; if $F$ is Archimedean we assume
that $E$ equals $\bbC$. Unless otherwise specified, all 
functions will take values in $E$.

\subsubsection{} Let $\C$ denote the space of $K$-finite $C^\infty$
functions on $F^n \setminus \{0\}$. In \S\ref{sect:C} we define the subspace 
$\Cp \subset \C$ consisting of functions with bounded support and the
subspace $\Cm \subset \C$ consisting of functions supported away from a neighborhood of $0$. 

\subsection{Subject of this article} In this article we consider the Radon transform 
as an operator \[ M : \Cp \to \Cm. \] 
When $F$ is non-Archimedean, $M$ is known to be an isomorphism \cite{BK}. 
An explicit formula for the inverse was, however, not present in the literature. There is a `classical' 
inversion formula due to \v{C}ernov \cite{Chernov} on the space of Schwartz functions, but its
relation to $M^{-1}$ is not obvious.
We formulate and prove a simple formula for $M^{-1}$ in the non-Archimedean case
and relate it to \v{C}ernov's formula. 

In the Archimedean case, the invertibility of $M$ was \emph{a priori} unclear due
to the nonstandard nature of the function spaces $\C_\pm$. We prove that $M$ is indeed
an isomorphism when $F$ is Archimedean and provide formulas for $M^{-1}$ (here $K$-finiteness
of $\eC$ plays a crucial role).

\subsection{Motivation}
Our interest in the operator $M$ originates from the classical theory of automorphic forms. 
Let $G$ denote the algebraic group $\SL_2$ and $N$ (resp.~$N^-$) the subgroup of strictly upper (resp.~lower) triangular matrices and $T$ the maximal torus of diagonal matrices. 
Then $G(F)/N(F) = F^2 \setminus \{0\}$ and $G(F)/N^-(F) = F^2\setminus \{0\}$. 
The operator $M$ is the standard (local) intertwiner $M:\Cp(G(F)/N(F)) \to \Cm(G(F)/N^-(F))$. 

While we work only with the local field $F$, one gets a global analog of $M$ by
considering the standard intertwiner $M : \Cp(G(\bbA)/T(F)N(\bbA)) \to \Cm(G(\bbA)/T(F)N(\bbA))$
where $F$ is a global field and $\bbA$ the adele ring. 
The intertwiner plays an important role in the theory of Eisenstein series and 
their constant terms \cite[\S 3.7]{Bump}. The constant terms
of automorphic forms reside in the space $\Cm(G(\bbA)/T(F)N(\bbA))$, which makes it a
natural space to study in this setting. The results of this article are used 
to prove invertibility of the global intertwiner in \cite{DW}. 

In the situation where $F$ is a non-Archimedean local field, the operator $M^{-1}$ is
essentially the same as the `Bernstein map' introduced in \cite[Definition 5.3]{BK}; 
the precise relation between the two is explained in \cite[Theorem 7.5]{BK}. 
The Bernstein map is also studied in \cite{SV} (there it is called the asymptotic map)
in the more general context of spherical varieties. 

In the real case, the Radon transform has been studied extensively
by analysts (\cite{Helgason94}, \cite{Helgason-GGA}, \cite{Helgason-Radon})
over slightly different function spaces. 

\subsection{Structure of the article}
In \S\ref{sect:Radon} we define the subspaces $\eC_\pm \subset \eC$ and 
recall the definition of the Radon transform over a general local field $F$. 

In \S\ref{sect:nonarch}, we consider the case when $F$ is non-Archimedean. 
We prove that $M$ is invertible and give a formula for $M^{-1}$ in Theorem~\ref{thm:nonarch}. 
This is done by relating the Radon transform to the Fourier transform (\S\ref{sect:fourier}-\ref{sect:RadF}). We deduce the previously known Radon inversion formula of \v{C}ernov \cite{Chernov} 
from Theorem~\ref{thm:nonarch} in \S\ref{sect:chernov}. 

We consider the real case in \S\ref{sect:real}. The formula for
$M^{-1}$ is given on each $K$-isotypic component of $\Cm$ in
Theorem~\ref{thm:real} in terms of convolution with a distribution
on $\bbR_{>0}$. The Mellin transform of this distribution is 
computed in Theorem~\ref{thm:convR}. The proof of the theorems is
in \S\ref{sect:proofR}. The invertibility of $M$ heavily relies on 
the $K$-finiteness assumption in the definition of $\C$. 
In \S\ref{sect:nonKfin}, we prove (Corollary~\ref{cor:notsurj}) 
that the analog of $M$ is not surjective when $K$-finiteness is dropped
from the definitions.

In \S\ref{sect:complex}, the complex case is developed in the same way 
as the real case. The inversion formula is given in Theorem~\ref{thm:complex} 
and the reformulation using the Mellin transform is Theorem~\ref{thm:convC}.

\subsection{Acknowledgments} 
The research was partially supported
by the Department of Defense (DoD) through the NDSEG fellowship.
I am very thankful to 
my doctoral advisor Vladimir Drinfeld for his continual guidance
and support throughout this project. 

\section{Recollections on the Radon transform}\label{sect:Radon}
\subsection{The norm on $F^n$} 
Let $\abs{\cdot}$ denote the normalized absolute
value on $F$ when $F$ is non-Archimedean and the usual absolute value\footnote{If $F=\bbR$, then
the normalized absolute value coincides with the usual absolute value. If $F=\bbC$, then 
the normalized absolute value is the square of the usual absolute value.} when $F$ is Archimedean. 
For $a \in F^\times$, set $v(a) := -\log \abs{a}$. If $F$ is
non-Archimedean $\log$ stands for $\log_q$, where $q$ is the order
of the residue field of $F$. If $F$ is Archimedean, $\log$ is understood as
the natural logarithm. 

We define a norm $\norm{\cdot}$ on $F^n$ as follows. 
If $F$ is non-Archimedean, then $\norm{\cdot}$ is the norm induced by 
the standard lattice $\eO^n$ (i.e., $\norm{x}$ is the
maximum of the absolute values of the coordinates of $x \in F^n$). 
If $F$ is Archimedean, then $\norm{\cdot}$ is induced by the standard 
Euclidean/Hermitian inner product (i.e., the square root of the 
sum of the absolute values squared). 

For $x\in F^n \setminus \{0\}$, set $v(x):= -\log \norm x$. 
 
\subsection{The spaces $\C, \C_c, \C_\pm$}  \label{sect:C}
Let $\C$ denote the space of $K$-finite $C^\infty$ functions on 
$F^n \setminus \{0\}$ (recall that if $F$ is non-Archimedean, 
$C^\infty$ means locally constant). 
Let $\C_c \subset \C$ be the subspace of compactly supported functions
on $F^n\setminus \{0\}$. 

Given a real number $R$, let $\C_{\le R}\subset\C$ denote the set
of all functions $\varphi\in \C$ such that $\varphi(\xi)\ne 0$ only if $v(\xi) \le R$. 
Similarly, we have $\C_{\ge R}, \C_{>R}$, and so on.
Let $\Cm$ denote the union of the subspaces $\C_{\le R}$ for all $R$. 
Let $\Cp$ denote the union of the subspaces $\C_{\ge R}$ for all $R$. 
Clearly $\Cm \cap \Cp = \C_c$ and $\C_- + \C_+ = \C$. 

\subsection{Radon transform}  \label{sect:radon}
Equip $F$ with the following Haar measure: if $F$ is non-Archimedean we require
that $\on{mes}(\eO)=1$; if $F$ is Archimedean we use the usual Lebesgue measure.
Let the measure on $F^n$ be the product of the measures on $n$ copies of $F$. 
Fix the Haar measure on $F^\times$ to be $d^\times t:= \frac{dt}{\abs t}$.

Let $f \in \Cp$. The Radon transform $\eR f(\xi, s)$,
for $\xi \in F^n \setminus \{0\}$ and $t \in F^\times$, is defined by 
the formula
\[ \eR f(\xi,t) = \int_{F^n} f(x) \delta(\xi \cdot x - t) dx, \]
where $\xi\cdot x = \xi_1 x_1 + \dotsb + \xi_n x_n$ and 
$\delta$ is the delta distribution on $F$.
The expression for $\eR f(\xi,t)$ can also be written directly as
\[ \eR f(\xi,t) = \int_{\xi \cdot x = t} f(x) d\mu_\xi \]
where $d\mu_\xi$ is the measure on the hyperplane $\xi \cdot x = t$ such that
$d\mu_\xi dt = dx$. We get an operator $M : \Cp \to \Cm$ by setting 
\begin{equation}\label{eqn:M} 
	M f(\xi) = \int_{\xi \cdot x = 1} f(x) d\mu_\xi. 
\end{equation}

\begin{prop} \label{prop:Mfil}
	For any number $R$ one has $M(\C_{\ge R}) \subset \C_{\le -R}$.  
\end{prop}
\begin{proof}
	Let $f \in \C_{\ge R}$ and $\xi \in F^n \setminus\{0\}$ with $v(\xi)
	> -R$. Then $\xi \cdot x = 1$ implies $v(x) <R$, so $f(x)=0$. 
	Therefore $Mf \in \C_{\le -R}$.
\end{proof}

\subsubsection{}
The natural action of $G$ on $F^n \setminus \{0\}$ induces
a $G$-action on $\C$ by $(g\cdot f)(x):=f(g^{-1}x)$ for $g \in G, f\in \C,
x\in F^n \setminus \{0\}$. Then 
\begin{equation}\label{eqn:Mequiv}
 M(g\cdot f) = \abs{\det g}^d (g^T)^{-1} Mf
\end{equation} 
for $f \in \Cp$ and $g \in G$, where $g^T$ is the transpose matrix,
and $d=1$ if $F\ne \bbC$ and $d=2$ if $F=\bbC$ (i.e., $\abs{\det g}^d$ is
the normalized absolute value of $\det g$).

\section{$F$ non-Archimedean} \label{sect:nonarch}
In this section we consider the case when $F$ is a non-Archimedean 
local field. Let 
$\eO$ the ring of integers, $\fp$ the maximal ideal, $\varpi$ a uniformizer, 
and $\bbF_q$ the residue field of $F$.

The main result of this section is Theorem~\ref{thm:nonarch}.
In order to state the theorem, we must first define a new operator
$A_\beta : \Cm \to \Cp$, which is done in \S\ref{sect:A}. 

\subsection{$K$-finite functions} The action of $G$ on $F^n \setminus \{0\}$
is continuous and transitive. 
Since $F$ is non-Archimedean, $K$ is an open subgroup of $G$, and 
we have the following description of $K$-finite functions. 

\begin{lem} \label{lem:Ksmooth}
A $C^\infty$ function $\varphi$ on $F^n \setminus \{0\}$ is $K$-finite
if and only if there exists an open subgroup $H \subset K$ 
such that $\varphi(h \xi)=\varphi(\xi)$ for all $h \in H$ and 
$\xi \in F^n \setminus \{0\}$.
\end{lem}
\begin{proof}
Let $W$ denote the span of the $K$ translates of $f$. By assumption 
$W$ is finite dimensional, and this implies that there exists a compact
open subset $X$ of $F^n \setminus \{0\}$ such that the restriction 
map $W \to C^\infty(X)$ is injective. Any locally constant 
function on $X$ is fixed by an open subgroup of $K$, which proves
the lemma. 
\end{proof}

One may sometimes wish to consider the group $\SL_n(F)$ rather than 
$\GL_n(F)$ acting on $F^n \setminus \{0\}$. The next lemma shows that 
this does not change the corresponding subspaces of invariant functions in $\C$.

\begin{lem} \label{lem:SLGL}
For an integer $r>0$, set $K_r:= \ker( \GL_n(\eO) \to \GL_n(\eO/\fp^r) )$.
Then the following properties of a function $\varphi$ on $F^n \setminus\{0\}$ are equivalent
for $n\ge 2$: 

(i) $\varphi$ is stabilized by $K_r\cap \SL_n(F)$,

(ii) $\varphi(\xi') = \varphi(\xi)$ for $\xi,\xi'\in F^n\setminus \{0\}$ satisfying $v(\xi'-\xi) \ge v(\xi) + r$,

(iii) $\varphi$ is stabilized by $K_r$.
\end{lem}

\begin{proof}
Suppose that $\varphi$ is stabilized by $K_r \cap G$. 
Take $\xi,\xi' \in F^n\setminus\{0\}$ with $v(\xi'-\xi) \ge v(\xi)+r$.
We can find a basis $v_1,\dotsc,v_n$ of $\eO^n$ with $v_1 = \varpi^{-v(\xi)}\xi$
and $v_2 = \varpi^{-v(\xi'-\xi)}(\xi'-\xi)$. 
Let $g$ send $v_1$ to $v_1+\varpi^{v(\xi'-\xi) - v(\xi)} v_2$ 
and $v_k$ to $v_k$ for $k>1$. 
Then $g \in K_r \cap G$ and $g \xi = \xi'$. Thus $\varphi(\xi')=\varphi(\xi)$.
This proves (i) implies (ii). 
The other implications are easy.
\end{proof}

\subsection{The operator $A_\beta : \Cm \to \Cp$} \label{sect:A}
Let $\eS'_b(F)$ denote the space of distributions $\beta$ on $F$ such that for any open subgroup $U \subset \eO^\times$,
the multiplicative $U$-average\footnote{The multiplicative $U$-average $\beta_U$ is defined
by $\beta_U(t) = \frac 1{\on{mes}(U)} \int_U \beta(u t) d^\times u$.} 
$\beta_U$ has compact support and $\brac{\beta_U,1}=0$.
Note that if $\brac{\beta_U,1}=0$ for some $U$, then it is true for all $U$.

\subsubsection{}
We would like to define $A_\beta : \Cm \to \Cp$ for $\beta \in \eS'_b(F)$
by \[ (A_\beta \varphi)(x) = \int_{F^n} \beta(\xi \cdot x) \varphi(\xi)d\xi \] 
but we must explain the meaning of the r.h.s. 

Fix $\varphi\in \Cm$ and $x \in F^n \setminus \{0\}$. 
For any open compact subgroup $\Lambda \subset F^n$ let 
\[ I(\Lambda) := \int_{\Lambda} \beta (\xi\cdot x)\varphi (\xi) d\xi. \] 

\begin{lem} \label{lem:AI}
There exists $\Lambda$ such that 
$I(\Lambda') = I(\Lambda)$ for any $\Lambda'$ containing $\Lambda$. 
\end{lem} 
\begin{proof} Choose $\xi_0 \in F^n$ such that $\xi_0 \cdot x =1$ and $v(\xi_0) = -v(x)$. 
Then $F^n = F\xi_0 \oplus H$ where $H$ is the hyperplane $\{ \xi \mid \xi \cdot x = 0\}$.
Lemma \ref{lem:Ksmooth} implies that $\varphi \in \Cm$ is fixed by 
the homothety actions of an open subgroup $U \subset \eO^\times$.
Therefore we can replace $\beta$ by the multiplicative average $\beta_U$. 
Let $\fp^i \subset F$ be a fractional ideal containing the support of $\beta_U$.
Lemma \ref{lem:SLGL}(ii) implies that 
$\varphi(s\xi_0 + \xi) = \varphi(\xi)$ if $s\in \fp^i$ and $v(\xi) \le v(\xi_0)+i-r$,
where $\varphi$ is stabilized by the congruence subgroup $K_r$. Put $a:= r-i$. 
Let $\Lambda := \fp^i \xi_0 \oplus \{\xi \in H \mid v(\xi) \ge -v(x)-a\}$. 

Now suppose $\Lambda'$ is a subgroup containing $\Lambda$. 
Define $\Lambda'' = \{ \xi \in \Lambda' \mid \xi\cdot x \in \fp^i \} \supset \Lambda$.
Then $I(\Lambda') = I(\Lambda'')$ since $\fp^i$ contains the support of $\beta$.
Now $\Lambda'' = \fp^i \xi_0 \oplus (\Lambda'' \cap H)$.
Thus 
\[ I(\Lambda'')-I(\Lambda) = \int_{\xi \in (\Lambda'' \setminus \Lambda)\cap H}
 \int_{\fp^i} \beta_U(s) \varphi(s\xi_0+\xi)\abs{\xi_0} ds d\mu_x. \] 
Note that $\xi \in (\Lambda''\setminus\Lambda)\cap H$ satisfies 
$v(\xi) < -v(x)-a$ and hence $\varphi(s\xi_0+\xi)=\varphi(\xi)$. 
We conclude that $I(\Lambda'')=I(\Lambda)$ since $\brac{\beta_U,1}=0$.
\end{proof}

\subsubsection{} 
Put $(A_\beta \varphi)(x) :=  I(\Lambda)$ where $\Lambda$ is as in Lemma \ref{lem:AI}.

\begin{cor} \label{cor:Afil}
Let $R$ be any number. 
If $\varphi \in \C_{\le -R}$, then $A_\beta \varphi \in \C_{\ge R-a}$,
where $a$ is an integer depending only on $\beta$ and the stabilizer of $\varphi$ in $G$. 
\end{cor}
\begin{proof} We use the notation from the proof of Lemma~\ref{lem:AI}. 
Note that the choice of $a$ is independent of $x \in F^n\setminus\{0\}$.
It follows from our definition above and the proof of Lemma~\ref{lem:AI} that
\[ (A_\beta\varphi)(x) = \int_{\substack{\xi \in H\\ v(\xi) \ge -v(x)-a}} 
\int_{\fp^i} \beta_U(t)\varphi(s\xi_0+\xi) \abs{\xi_0} ds d\mu_x, \]
which is zero if $v(x)<R-a$. 
\end{proof}

Thus we have defined an operator $A_\beta : \Cm \to \Cp$.


\begin{rem}
For $\varphi \in \Cm$ we have 
$A_\beta(g\varphi) = \abs{\det g} (g^T)^{-1}(A_\beta\varphi)$
where $g^T$ is the transpose. In other words, the operator 
$\Cm \to \{\text{measures on }(F^n)^*\setminus \{0\}\}$ defined by
$\varphi \mapsto (A_\beta \varphi)dx$ is equivariant with respect to the
action of $G$.
\end{rem}

The goal of this section is to prove the following.
\begin{thm}\label{thm:nonarch}
The operator $M : \Cp \to \Cm$ is an isomorphism. 
The inverse of $M$ is $A_\beta$, where $\beta$ is the compactly supported
distribution on $F$ equal to
	\[ \frac{1-q^{n-1}}{1-q^{-n}} (\abs{s-1}^{-n} - \abs{s}^{-n}). \]
\end{thm}

The distributions $\abs{s-1}^{-n}$ and $\abs s^{-n}$ are defined as in \cite[Ch.~2, \S 2.3]{Gelfand-6}, i.e., 
\[ \brac{\abs s^{-n}, f}= \int_F \abs s^{-n} (f(s)-f(0)) ds \]
for a test function $f \in C^\infty_c(F)$.

We prove Theorem~\ref{thm:nonarch} in \S\ref{sect:RadF}.

\begin{rem} Let $\beta$ be as defined in Theorem~\ref{thm:nonarch}.
Then the integral of $\beta$ along any compact open subset of $F$ 
has value in $\bbZ[\frac 1 q]$. This is not true for the distribution 
$\frac{1-q^{n-1}}{1-q^{-n}}\abs{s}^{-n}$.
\end{rem}

\subsection{Fourier transform} \label{sect:fourier}
We assume without loss of generality that $E$ contains all roots of unity. 
Choose a nontrivial additive character $\psi$ of $F$ which is trivial on $\eO$ but
nontrivial on $\varpi^{-1}\eO$. The Haar measure we chose for $F$ is self-dual
with respect to $\psi$. 
Note that $\psi \in \eS'_b(F)$. Define the Fourier transform 
$\eF: \Cm \to \Cp$ by 
\[ \eF := A_\psi. \] 

\noindent
On the other hand, we also have an operator $\eF' : \Cp \to \Cm$ defined by 
\begin{equation} \label{eqn:F'}
	\eF' f(\xi) = \int_{F^n} f(x) (\psi(-\xi\cdot x)-1)dx. 
\end{equation}
Moreover for any number $R$, one observes that $\eF'(\C_{\ge R}) \subset \C_{< -R}$.

\begin{prop}\label{prop:fourier}
	The operators $\eF$ and $\eF'$ are mutually inverse. \end{prop}
\begin{proof}
Proposition~\ref{prop:Mfil} and Corollary~\ref{cor:Afil} imply
that $\eF \eF' (\C_{\ge R}) \subset \C_{\ge R+a}$ and $\eF'\eF(\C_{\le -R}) \subset \C_{\le -R-a}$
on functions stabilized by $K_r$ for a fixed $r>0$. As a consequence, it is
enough to check the equalities $\eF \eF' = \id$ and $\eF' \eF = \id$
on the subspace $\C_c = \Cp \cap\Cm$. 

Let $f \in \C_c$. Then the usual Fourier transform $\hat f$ is a compactly
supported function on $F^n$.
Note that $\eF' f(\xi) = \hat f(\xi) - \hat f(0)$. By the definition of $\eF$, we have
\[ \eF \eF' f(x) = \int_\Lambda (\hat f(\xi)-\hat f(0))\psi(\xi \cdot x) d\xi \]
for any sufficiently large open compact subgroup $\Lambda \subset F^n$.
Since $\hat f$ is compactly supported, the usual Fourier inversion formula implies that
$\int_\Lambda \hat f(\xi) \psi(\xi \cdot x) d\xi = f(x)$ if $\Lambda$ contains
the support of $\hat f$. Since $x$ is nonzero, $\int_\Lambda \psi(\xi \cdot x)d\xi =0$
for $\Lambda$ large enough. Therefore $\eF \eF' f = f$.
\smallskip

In the other direction, let $\varphi \in \C_c$. Then
$\eF \varphi(x) = \hat\varphi(-x)$ is compactly supported on $F^n$. Again the Fourier inversion formula implies that
\[ \eF' \eF \varphi(\xi) = \int_{F^n} \hat\varphi(x)\psi(-\xi \cdot x)dx - 
\int_{F^n} \hat\varphi(x) dx = 
\varphi(\xi) - \varphi(0) = \varphi(\xi). \qedhere \]
\end{proof}

\subsection{Actions on $\C_\pm$} 
For any real number $a$, let $\eA_{\le a}$ be the space of generalized functions  
$\alpha$ on $F^{\times}$ whose support is contained in
$\{ t \in F^\times \mid v(t)\le a \}$. 
Let $\eA_-$ denote the union of all $\eA_{\le a}$ for all $a$. Then $\eA_-$ becomes an algebra 
under convolution using the measure $d^\times t$. 

\subsubsection{} 
We have an action of $\eA_-$ on $\Cm$ defined by 
\[ (\alpha * \varphi )(\xi) = \int_{F^\times} \alpha(t) \varphi(t^{-1}\xi) d^\times t \]
for $\alpha \in \eA_-,\, \varphi \in \Cm$, and $\xi \in F^n\setminus\{0\}$.
One similarly defines $\eA_{\ge a},\, \eA_+$, and an action of $\eA_+$ on $\Cp$.
There is an isomorphism $\sigma : \eA_{\le a} \to \eA_{\ge -a}$ defined by 
\[ \sigma(\alpha)(t) = \alpha(t^{-1}) \abs{t}^{-n}. \] 

\subsubsection{}
We would like to define a multiplicative convolution action of $\eA_+$ on $\eS'_b(F)$ by 
\[ (\wtilde \alpha * \beta)(s) = \int_{F^\times} \wtilde \alpha(t) \beta(t^{-1}s) d^\times t \]
for $\wtilde \alpha \in \eA_+$ and $\beta \in \eS'_b(F)$, but we must explain 
the meaning of this formula as a distribution on $F$.  
Let $\eS(F)$ denote the space of locally constant, compactly
supported functions on $F$.

\begin{lem} \label{lem:schwartz} 
Let $f \in \eS(F)$ and $t \in F^\times$. Then $\int_F \beta(t^{-1} s) f(s) ds = 0$ if $v(t)$ is sufficiently large. 
\end{lem}
\begin{proof}
Since $f\in \eS(F)$, there exists an open subgroup $U\subset \eO^\times$ that stabilizes $f$
under homotheties. Thus we can replace $\beta$ by the multiplicative average $\beta_U$,
which is compactly supported. 
Then $\int_F \beta(t^{-1}s) f(s) ds = \abs t\int_{\supp \beta_U} \beta_U(s) f(ts)ds$. 
If $v(t)$ is large enough such that $f$ is constant on $t(\supp \beta_U)$, the integral
vanishes since $\brac{\beta_U,1} = 0$.
\end{proof}

Define the distribution $\wtilde \alpha * \beta \in \eS'_b(F)$ by putting the value
at $f \in \eS(F)$ to be 
\[ \brac{\wtilde \alpha * \beta, f} = \int_{F^\times} \wtilde \alpha(t) 
\left( \int_F \beta(t^{-1} s) f(s) ds \right) d^\times t, \]
which is well-defined by Lemma~\ref{lem:schwartz} and the fact that $\wtilde \alpha \in \eA_+$.

\begin{rem} \label{rem:*filt}
Observe that $\eA_{\le a} * \C_{\le R} \subset \C_{\le R+a}$ 
and $\eA_{\ge a} * \C_{\ge R} \subset \C_{\ge R+a}$ for any numbers $a$ and $R$.
Moreover if $\wtilde \alpha \in \eA_{\ge a}$ and $\beta \in \eS'_b(F)$ has support contained in 
$\fp^i$, then the support of $\wtilde \alpha * \beta$ is contained in $\fp^{a+i}$.
\end{rem}

\begin{rem}
The convolution action of $\eA_+$ on $\eS'_b(F)$ is indeed an action, i.e., 
$\wtilde \alpha_1 * (\wtilde \alpha_2 * \beta) = 
(\wtilde \alpha_1 * \wtilde \alpha_2) * \beta$ for $\wtilde \alpha_1,\wtilde \alpha_2
\in \eA_+$ and $\beta \in \eS'_b(F)$. One sees this by restricting $\beta$
to $F^\times$ and identifying $\eA_+$ with the space of distributions on 
$F^\times$ with bounded support using the measure $d^\times t$.
\end{rem}

\begin{lem} \label{lem:A*}
Let $\alpha \in \eA_-$, $\beta \in \eS'_b(F)$, and $\varphi \in \Cm$. Then 
\[ A_\beta(\alpha * \varphi) = \sigma(\alpha) * A_\beta\varphi 
= A_{\sigma(\alpha)*\beta}(\varphi) \] 
\end{lem}
\begin{proof}
By Corollary~\ref{cor:Afil} and Remark~\ref{rem:*filt}, we reduce to the case 
where $\alpha \in \eA_- \cap \eA_+$ and $\varphi \in \C_c$. 
Consequently, $\alpha * \varphi \in \C_c$. Fix $x \in F^n \setminus \{0\}$. 
We have
\[ A_\beta(\alpha * \varphi)(x) = \int_{F^n}
\beta(\xi\cdot x) \int_{F^\times} \alpha(t) \varphi(t^{-1}\xi) d^\times t d\xi
= \int_{F^\times} \alpha(t)\abs{t}^n \int_{F^n} \beta(\xi \cdot tx)\varphi(\xi) d\xi d^\times t.
\]
by a change of variables.
Substituting $t$ with $t^{-1}$ in the last integral shows that 
$A_\beta(\alpha * \varphi) = \sigma(\alpha) * A_\beta\varphi$. One observes that
$\sigma(\alpha)*A_\beta\varphi = A_{\sigma(\alpha)*\beta}(\varphi)$ essentially by 
definition.
\end{proof}

\begin{rem}\label{rem:M*} 
One easily checks that if $\alpha \in \eA_-$ and $f\in \Cp$, then 
$M(\sigma(\alpha)* f) = \alpha* Mf$. 
\end{rem}

\subsection{Relation between Radon and Fourier transforms} \label{sect:RadF}
Note that $\eF'$ and $M$ are both operators $\Cp \to \Cm$. 
Comparing formulas
\eqref{eqn:M} and \eqref{eqn:F'}, we deduce the formula
\begin{equation} \label{eqn:F'R}
	\eF' f = \alpha * Mf
\end{equation}
where $f \in \Cp$ and $\alpha(t):= \psi(-t)-1$ for $t \in F^\times$. 

Let $\beta$ be the distribution defined in Theorem~\ref{thm:nonarch}. 

\begin{lem} \label{lem:keybeta}
	We have an equality of distributions
\[ \beta = \sigma(\alpha) * \psi. \]
\end{lem}
\begin{proof}
Let $f \in \eS(F)$. Then $\brac{\sigma(\alpha)*\psi, f} = 
\int_{F^\times} \abs{t}^n (\psi(t)-1) \left(\int_F f(s)\psi(-ts) ds \right) d^\times t$.
This is the value at $f$ of the Fourier transform of $\abs{t}^{n-1} (\psi(t)-1)$
considered as a distribution on $F$. It is well-known \cite[Ch.~2, \S 2.5-6]{Gelfand-6} that 
the Fourier transform of $\abs{t}^{n-1}$ is 
$\frac{1-q^{n-1}}{1-q^{-n}} \abs{s}^{-n}$. Therefore we conclude 
that $\sigma(\alpha)*\psi = \beta$.
\end{proof}

Observe that $\eF=A_\psi$ and $A_\beta$ are both operators $\Cm \to \Cp$. 
Let $\varphi \in \Cm$.
From Lemmas~\ref{lem:A*} and \ref{lem:keybeta}, we deduce the equality 
\begin{equation}\label{eqn:M'F}
	A_\beta\varphi = \sigma(\alpha)* \eF \varphi. 
\end{equation}

\begin{proof}[Proof of Theorem \ref{thm:nonarch}]
We deduce from \eqref{eqn:F'R} and Proposition~\ref{prop:fourier} that
$M$ has a left inverse sending $\varphi\in \Cm$ to 
$\eF(\alpha*\varphi)$. Lemma~\ref{lem:A*} and \eqref{eqn:M'F} together say that 
$\eF(\alpha*\varphi)= A_\beta\varphi$.
Applying $M$ to \eqref{eqn:M'F} and using Remark~\ref{rem:M*}, 
we see that $MA_\beta = \eF'\eF = \id$. Therefore $A_\beta$ is both left and right
inverse to $M$.
\end{proof}

\subsection{Comparison with \v{C}ernov's Radon inversion formula} 
\label{sect:chernov}
Let $f$ be a Schwartz (i.e., compactly supported $C^\infty$) function on $F^n$.
Recall that the Radon transform $\eR f(\xi,s)$ is a $C^\infty$ function on 
$(F^n \setminus \{0\}) \xt F$ (in particular
it is defined at $s=0$), and $\eR f(\xi,s)=0$ if $\norm{s\xi}^{-1}$ is sufficiently large.
The following ``non-archimedean Cavalieri's condition'' is also well-known:
\begin{lem} \label{lem:cavalieri} 
	The integral $\int_F \eR f(\xi,s) ds$ does not depend on $\xi$.
\end{lem}
\begin{proof}
The integral of $f$ over $F^n$ along a pencil of parallel hyperplanes does not depend on the direction of the pencil.
\end{proof}

\subsubsection{}
It was previously known (\cite[Theorem 5]{Chernov}, \cite[formula (8)]{Kochubei}) 
that the following inversion formula holds:
\begin{equation} \label{eqn:k8}
f(x) = \frac {1-q^{n-1}}{(1-q^{-1})(1-q^{-n})} \int_{\norm\eta = 1} 
\brac{\abs s^{-n}, \eR f(\eta, s+\eta \cdot x)} d\eta
\end{equation}
where $x \in F^n \setminus \{0\}$ and $\eta$ ranges over norm $1$ vectors in $F^n$.

\subsubsection{}
We will deduce formula \eqref{eqn:k8} from Theorem~\ref{thm:nonarch}. 
Since $f$ is compactly supported on $F^n$, we have 
$f \in \Cp$ and Theorem~\ref{thm:nonarch} implies that 
\[ f(x)=A_\beta M f (x)= \int_{v(\xi) \ge R} \beta(\xi \cdot x) Mf(\xi) d\xi  \]
for $x \in F^n\setminus \{0\}$ and $R$ a sufficiently large number. 
We can write $\xi = t^{-1} \eta$ where $t \in F^\times$
and $\eta \in F^n$ with $\norm\eta = 1$. This gives the equality
\[ f(x) = \int_{v(t) \le -R} \int_{\norm\eta = 1} \beta(t^{-1}\eta\cdot x) Mf(t^{-1}\eta)\abs{t}^{-n} d\eta d^\times t. \]
Homogeneity of $\eR f$ implies that $\abs{t}^{-1} Mf(t^{-1}\eta) = \eR f(\eta, t)$.
Therefore we have the formula
\begin{equation}\label{eqn:AM} 
	f(x) = \frac {1-q^{n-1}}{(1-q^{-1})(1-q^{-n})}\int_{v(t)\le -R} \int_{\norm\eta=1}
(\abs{\eta \cdot x - t}^{-n} - \abs{\eta\cdot x}^{-n}) \eR f(\eta,t) d\eta dt. 
\end{equation}
Choose $\eta_0 \in F^n$ with $v(\eta_0)=-v(x)$ and $\eta_0 \cdot x = 1$. 
Then $\eta \cdot x - t = (\eta-t \eta_0)\cdot x$. Note that if $v(t)>v(x)$, then
translation by $t\eta_0$ preserves the unit sphere of norm $1$ vectors.
Moreover smoothness of $\eR f$ implies that $\eR f(\eta+t\eta_0,t) = \eR f(\eta,t)$ 
if $v(t)$ is sufficiently large.
Therefore the inner integral of \eqref{eqn:AM} is zero if $v(t)$ is sufficiently large. 
Thus we may integrate over all $t \in F$ and switch the order of integration. 

\begin{lem} \label{lem:Kochubei}
The integral $\int_{\norm\eta = 1} \abs{\eta \cdot x}^{-n} d\eta$ equals zero. 
\end{lem} 
\begin{proof}
Using the $G$-action, we may assume that $x=(1,0,\dotsc,0)$. Then
$\eta \cdot x = \eta_1$, the first coordinate of $\eta$. 
One sees that $\int_{\norm\eta = 1} \abs{\eta_1}^{-n} d\eta = 
(1-q^{-1}) + \int_\fp \abs{\eta_1}^{-n}d\eta_1 (1-q^{1-n})$. 
A simple calculation shows that the latter expression vanishes.
\end{proof}

Lemmas~\ref{lem:cavalieri} and \ref{lem:Kochubei} imply that 
$\int_{\norm\eta = 1} \abs{\eta \cdot x}^{-n} \int_F \eR f(\eta,t)dt d\eta = 0$,
so the $\abs{\eta \cdot x}^{-n}$ term in \eqref{eqn:AM} vanishes.
After a change of variables $s = t-\eta \cdot x$, the formula \eqref{eqn:AM}
becomes equal to \v{C}ernov's formula \eqref{eqn:k8}.

\section{$F$ real} \label{sect:real}
In this section we prove the invertibility of $M$ when $F=\bbR$.
Recall that in this case $K = O(n)$. 
The inversion formula is given in Theorem~\ref{thm:real},
and a reformulation using the Mellin transform is given in Theorem~\ref{thm:convR}. The $K$-finiteness of $\C$ plays a crucial role in the proofs,
so we begin by recalling the classification of the $K$-isotypic components 
of $\C$.

The non-$K$-finite situation is considered in \S\ref{sect:nonKfin}.

\subsection{Spherical harmonics} 
Let $S^{n-1}$ denote the unit sphere centered at the origin in $\bbR^n$,
which has a natural action by $O(n)$.
Let $\C(S^{n-1})$ be the space of smooth $K$-finite functions on $S^{n-1}$.
For a nonnegative integer $k$, let $H^k$ denote the space of harmonic 
polynomials on $\bbR^n$ of degree $k$. 

\begin{thm}[{\cite[Theorem I.3.1]{Helgason-GGA}, \cite[Theorem 3.1]{JW}, \cite{Kostant}}] 
\label{thm:sphR}

Let $H^k|S^{n-1}$ denote the space of harmonic polynomials 
restricted to $S^{n-1}$. Then 

(i) the restriction map $H^k \to H^k|S^{n-1}$ is an isomorphism,

(ii) $\C(S^{n-1}) = \bigoplus_{k \ge 0} H^k|{S^{n-1}}$ as $O(n)$-representations,

(iii) the $O(n)$-representations $H^k$ are irreducible and not isomorphic to each other.
\end{thm}

\subsection{Decomposing $\eC$ into $K$-isotypes}

We have a decomposition $\bbR^n\setminus \{0\} = \bbR_{>0} \xt S^{n-1}$,
with $O(n)$ acting on the $S^{n-1}$ component.
Let $\C(\bbR_{>0})$ denote the space of smooth functions on $\bbR_{>0}$
and define the subspaces $\C_\pm(\bbR_{>0}), \C_c(\bbR_{>0})$ as in \S\ref{sect:C}.

Theorem~\ref{thm:sphR} implies that there is a decomposition
\[ \C = \bigoplus_{k\ge 0} \C(\bbR_{>0}) \ot H^k. \]
For $u \in \C(\bbR_{>0})$ and $Y \in H^k$, we define $u\ot Y \in \C$ by
$(u \ot Y)(x) := u(\abs x) \cdot Y(\frac x{\abs x})$. 

\subsection{Radon inversion formula}
\subsubsection{}
We have an isomorphism $\Inv:\Cm \to \Cp$ defined by 
\[ (\on{Inv} \varphi)(x) = \norm x^{-n} \varphi\left(\frac{x}{\norm x^2}\right). \]
Set $\wtilde M := \on{Inv}^{-1} \circ M$.  
Consider $\bbR_{>0}$ as a subgroup of diagonal matrices in $G$. 
Then it follows from \eqref{eqn:Mequiv} that $\wtilde M$ is a $K\xt \bbR_{>0}$ equivariant operator from $\Cp$ to $\Cp$. 

\subsubsection{} \label{subsect:A}
Let $A$ denote the space of distributions on $\bbR_{>0}$ supported on $(0,1]$. 
Then $A$ is an algebra under the convolution product $*$ 
induced by the multiplication operation on $\bbR_{>0}$.
The action of $\bbR_{>0}$ on $\Cp$ induces an action of $A$.

\begin{thm} \label{thm:real}
The operator $M : \Cp \to \Cm$ is an isomorphism. For $\varphi \in \Cm(\bbR_{>0}) \ot H^k$,
the inverse $M^{-1} : \Cm \to \Cp$ is given by the formula
\[	M^{-1}\varphi = \beta_k * \Inv(\varphi) \]
where $\beta_k$ is the distribution on $\bbR_{>0}$ defined by 
\begin{equation}
\label{eqn:beta-real}
\beta_k(t) = \frac 1{2^{n+k-2} \pi^{\frac{n-1}2} \Gamma(\frac{n+2k-1}2)} t^{k-1} \left(-\frac d{dt}\right)^{n+k-1} \left(t^{-k+1} (1-t^2)_+^{\frac{n+2k-3}2} \right) dt.
\end{equation}
\end{thm}

\noindent
The derivative $\frac d{dt}$ is applied in the sense of generalized functions. 
For $\lambda \in \bbC$ with $\on{Re}(\lambda)>-1$, the 
generalized function $(1-t)_+^\lambda$ 
is defined by $\brac{(1-t)_+^\lambda , f_0(t)dt} = \int_0^1 (1-t)^\lambda f_0(t)dt$
for $f_0 \in \C_c(\bbR_{>0})$. This generalized function can be
analytically continued to all $\lambda \in \bbC$ not equal to a negative
integer \cite[\S I.3.2]{Gelfand-1}. We define 
$(1-t^2)_+^\lambda = (1+t)^\lambda \cdot (1-t)_+^\lambda$.

\begin{cor} For any number $R$ one has $M^{-1}(\C_{\le -R})
\subset \C_{\ge R}$. \end{cor}
\begin{proof}
Observe that $\beta_k$ is supported on $(0,1]$ for all $k$.
\end{proof}

\subsection{A formula for $\wtilde M$ in terms of convolution}
For $t \in (-1,1)$, define $A_t : \C(S^{n-1}) \to \C(S^{n-1})$
such that $(A_t f)(x)$ is the average value of $f$ on the $(n-2)$-sphere
$\{ \omega \in S^{n-1} \mid \omega \cdot x = t \}$.
Then $A_t$ is $O(n)$-equivariant, so by Schur's lemma
it acts on $H^k|S^{n-1}$ by a scalar $a_k(t)$. 
Since $H^k$ is stable under complex conjugation, $a_k(t)$ is real valued. 
One observes that $a_k$ is a smooth function on $(-1,1)$, 
$\abs{a_k(t)} \le 1$ for all $t \in (-1,1)$, 
and $\lim_{t\to 1} a_k(t)=1$.

Suppose that $f \in \C(\bbR_{>0}) \ot H^k$ and there exists $C>0$ and 
$\sigma > n-1$ such that $\abs{f(x)} \le C\norm{x}^{-\sigma}$ for all $x$ with $\norm x \ge 1$.
Since the intersection of $S^{n-1}$ with the hyperplane $\{ \omega \mid \omega \cdot x = t\}$ 
has radius $(1-t^2)^{1/2}$ for a unit vector $x$, we deduce that 
\begin{equation} \label{eqn:alphak}
 \wtilde M f = \alpha_k * f 
\end{equation}
where $\alpha_k$ is the measure $\on{mes}(S^{n-2}) \cdot t^{-n} \cdot a_k(t) (1-t^2)^{\frac{n-3}2}dt$
on the interval $(0,1)$ extended by zero to the whole $\bbR_{>0}$. 
The convolution $\alpha_k * f$ is well-defined because of the bound on $\abs{f(x)}$, and  
$\on{mes}(S^{n-2})$ denotes the surface area of the $(n-2)$-sphere.
In fact, \cite[Proposition 2.11]{Helgason94} says that $a_k(t)$ is the scalar
multiple of the Gegenbauer polynomial $C_k^{(\frac{n-2}2)}(t)$ normalized by $a_k(1)=1$.

The Mellin transform $\fM \alpha_k$ is defined for $s\in \bbC$ by integrating
$t^s$ against $\alpha_k$ if $\on{Re}(s) > n-1$. 

\begin{thm} \label{thm:convR}
The distribution $\alpha_k$ is invertible in $A$. The inverse $\beta_k$ is defined by
\eqref{eqn:beta-real}. The Mellin transforms are given by 
\begin{equation}\label{eqn:mellinR}
\fM \beta_k(s) = 
\frac 1{\fM\alpha_k(s)} = 2^{1-n-k} \pi^{\frac{1-n}2} \frac{\Gamma(s+k)}{\Gamma(s-n+1)} \cdot \frac{\Gamma(\frac{s-n-k}2+1)}{\Gamma(\frac{s+k+1}2)}.
\end{equation}
\end{thm}
\noindent Theorem~\ref{thm:convR} implies Theorem~\ref{thm:real}.

\subsection{Relation to Fourier transform} \label{sect:real-fourier}

Let $\eS(\bbR^n)$ denote the space of Schwartz functions on 
$\bbR^n$ and $\eS'(\bbR^n)$ the dual space of tempered distributions 
on $\bbR^n$. 
The Fourier transform is defined for an integrable function $f$ on $\bbR^n$ by 
\[ \eF f(\xi) = \int_{\bbR^n} f(x) e^{-2\pi i \xi \cdot x} dx. \]
This definition can be extended \cite[\S I.3]{SW} to the space 
of tempered distributions. After this extension, $\eF$ becomes
an isomorphism $\eF : \eS'(\bbR^n) \to \eS'(\bbR^n)$ .

Let $F : \eS'(\bbR) \to \eS'(\bbR)$ denote the $1$-dimensional Fourier transform.
For $f \in \eS(\bbR^n)$, one gets the Fourier transform from the Radon transform by
\begin{equation}\label{eqn:F-Rf}
	\eF f(r\omega) = F(\eR f(\omega,t))(r) 
\end{equation}
where $F$ is the Fourier transform with respect to the $t$ variable, 
$r \in \bbR$, and $\omega \in S^{n-1}$ is a unit vector.

\begin{lem} \label{lem:Rdist}
Let $f$ be a locally integrable function
on $\bbR^n$ for which there exist $C>0$ and $\sigma >n-1$ such that 
$\abs{f(x)}\leq C \norm{x}^{-\sigma}$ for all $x$ with $\norm x\ge 1$.
Then: 

(i) $\eR f$ is a locally integrable function on $S^{n-1} \xt \bbR$.

(ii) $\eR f(\omega,t)$ is bounded for $\abs t \ge 1$.

(iii) The right hand side of \eqref{eqn:F-Rf} is well-defined as a 
generalized function on $\bbR \xt S^{n-1}$.

(iv) Equation \eqref{eqn:F-Rf} holds as an equality between
generalized functions on $\bbR_{>0} \xt S^{n-1}$. 
\end{lem}

\begin{proof}
Since $\eR f$ is defined by integrating $f$ on a hyperplane 
of dimension $n-1$, the bound on $\abs{f(x)}$ implies that $\eR f$
is well-defined on $S^{n-1} \xt \bbR$. One also uses this bound
and local integrability of $f$ to deduce that $\eR f$ is locally integrable.
If $\omega \in S^{n-1}$ and $t\in \bbR$ with $\abs t \ge 1$, then integrating 
in the radial
direction on the hyperplane $\omega \cdot x = t$, we see that
$\abs{\eR f(\omega,t)}$ is bounded by a constant times 
$\int_0^\infty (r^2+t^2)^{-\sigma/2} r^{n-2}dr$, which 
is equal to a constant times $\abs{t}^{n-1-\sigma}$. This proves (ii). 
Property (iii) follows immediately from properties (i)-(ii).

Let $\varphi$ be a compactly supported smooth function on 
$\bbR^n\setminus\{0\} = \bbR_{>0}\xt S^{n-1}$. 
Consider $f$ as a tempered distribution on $\bbR^n$. 
By the definition of $\eF f$, 
\begin{equation}\label{eqn:Ffphi} 
	\int_{\bbR_{>0}\xt S^{n-1}} \eF f(r\omega) \varphi(r \omega) r^{n-1} dr d\omega = 
\int_{\bbR^n} f(x) \int_{\bbR_{>0}\xt S^{n-1}} \varphi(r\omega) e^{-2\pi i r (\omega \cdot x)} r^{n-1} dr d\omega dx. 
\end{equation}
Since $t\mapsto \int_{\bbR_{>0}} \varphi(r\omega) e^{-2\pi i r t} r^{n-1}dr$ is a Schwartz
function on $\bbR$, we deduce from the decomposition $dx = d\mu_\omega dt$ and 
property (ii) applied to $\abs f$ that the integral 
\[ \int_{S^{n-1}}\int_{\bbR^n} \left\lvert f(x)\int_{\bbR_{>0}} \varphi(r\omega) e^{-2\pi i r(\omega \cdot x)} r^{n-1} dr \right\rvert dx d\omega \] 
converges. Then the Fubini-Tonelli theorem implies that \eqref{eqn:Ffphi} is equal to
\[ \int_{S^{n-1}} \int_\bbR \int_{\bbR_{>0}} \eR f(\omega,t) e^{-2\pi i rt} \varphi(r\omega) r^{n-1} dr dt d\omega, \]
which proves (iv).
\end{proof}

\subsection{Proof of Theorem~\ref{thm:convR}} \label{sect:proofR}
Let $Y \in H^k$ and define $f(x) = \norm{x}^{-s} \cdot Y(\frac x {\norm x})$ for $s \in \bbC$. If $n-1 < \on{Re}(s) < n$, then $f$ is locally
integrable on $\bbR^n$ and satisfies the hypothesis of 
Lemma~\ref{lem:Rdist}. 
Moreover by \eqref{eqn:alphak} and homogeneity of $f$
we see that 
\begin{equation*}\label{eqn:Rf-M}
	\eR f(\omega, t) = \on{sgn}(t)^k \abs{t}^{n-1-s} Mf(\omega)
	= \on{sgn}(t)^k \abs{t}^{n-1-s} \fM \alpha_k(s) Y(\omega)
\end{equation*}
as a locally integrable function on $S^{n-1} \xt \bbR$.
Then Lemma~\ref{lem:Rdist} implies that 
\[ \eF f(r\omega) = F(\on{sgn}(t)^k \abs{t}^{n-1-s})(r)\fM\alpha_k(s) Y(\omega). \]
It is well-known \cite[\S II.2.3]{Gelfand-1} that
\[ F(\on{sgn}(t)^k \abs{t}^{n-1-s})(r) = i^{k} (2\pi)^{s-n+1}\frac{\sin(\pi\frac{(s-n-k+1)}2)}\pi \Gamma(n-s) \on{sgn}(r)^k \abs r^{s-n}. \]

\noindent
On the other hand, one can compute the Fourier transform of $f$ directly: 
\begin{thm}[{\cite[Theorem IV.4.1]{SW}}] \label{thm:SW}
If $0 < \on{Re}(s) < n$, then 
$\eF f(x) = \gamma \norm{x}^{s-n} Y(\frac x{\norm x})$, where 
$\gamma = i^{-k} \pi^{s-\frac n 2} \Gamma(\frac{n+k-s}2)/\Gamma(\frac{s+k}2)$.
\end{thm}

\noindent
Comparing constant multiples in the two formulas for $\eF f$ above and applying Euler's reflection formula, we have 
\[ \fM\alpha_k(s) 
= 2^{n-1-s} \pi^{n/2-1} \frac{\Gamma( \frac{n+k-s}2) \Gamma(\frac{s-n-k+1}2)
\Gamma(\frac{n+k+1-s}2)}{\Gamma(\frac{s+k}2)\Gamma(n-s)}
\]
for $n-1 < \on{Re}(s) < n$. By analytic continuation, we deduce the
equality for all $s \in \bbC$ away from poles.
The duplication formula for the $\Gamma$-function implies that
\begin{equation} \label{eqn:Ma}
	\fM\alpha_k(s) = 2^{n+k-1} \pi^{\frac{n-1}2} \frac{ \Gamma(s-n+1)}{\Gamma(s+k)} \cdot \frac{\Gamma(\frac{s+k+1}2)}{\Gamma(\frac{s-n-k}2+1)},
\end{equation} 
as stated in Theorem~\ref{thm:convR}. To finish the proof of Theorem~\ref{thm:convR}, it remains to show that $(\fM\beta_k)^{-1}$ equals the right hand side of \eqref{eqn:Ma}.
By considering the Beta function we see that
$\Gamma(\frac{s-n-k}2+1)/\Gamma(\frac{s+k+1}2)$ is the Mellin transform of $\nu(t)dt$, where 
\begin{equation}\label{eqn:nu}
 \nu(t) = \frac 2{\Gamma(\frac{n+2k-1}2)} t^{1-n-k} (1-t^2)_+^{\frac{n+2k-3}2}. 
\end{equation}
The generalized function $(1-t^2)^{\frac{n+2k-3}2}_+$ is defined in the paragraph after Theorem~\ref{thm:real}.
Multiplying the right hand side of \eqref{eqn:nu} by 
$\Gamma(s+k)/\Gamma(s-n+1) = (s-n+1)\dotsb (s+k-1)$ amounts to 
replacing $\nu$ by $L_k(\nu)$, where 
$L_k := (-\frac d{dt}\cdot t -n+1)\dotsb(-\frac d{dt}\cdot t +k-1)$. 
Observe that $L_k = t^{k-1} (-\frac d{dt})^{n+k-1} t^n$. Therefore $\fM \beta_k = (\fM \alpha_k)^{-1}$, where $\beta_k$ is defined by \eqref{eqn:beta-real}.
This proves Theorem~\ref{thm:convR}.

In the case $n=2$, the formula \eqref{eqn:Ma} is well-known (cf.~\cite[Lemma 7.17]{Wallach}, \cite[Proposition 2.6.3]{Bump}). 

\subsection{The non-$K$-finite situation} \label{sect:nonKfin}
In this subsection we consider the situation where we remove
$K$-finiteness from the definitions of $\Cp$ and $\Cm$. 
Let $\CCp$ be the space of smooth functions on $\bbR^n\setminus\{0\}$
with bounded support, and let $\CCm$ be the space of
smooth functions on $(\bbR^n)^* \setminus \{0\}$ supported away from 
a neighborhood of $0$.

\subsubsection{}
We have the operator $\sM : \CCp \to \CCm$ defined by 
\[ \sM f(\xi) = \int_{\brac{\xi,x}=1} f(x)d\mu_\xi \]
(cf.~formula \eqref{eqn:M}). 
One can deduce that $\sM$ is injective from the injectivity of 
$M : \Cp \to \Cm$. However we will show below that $\sM$ is not surjective,
and hence not an isomorphism.

\subsubsection{} Let $f \in \CCp$.  
Define $C_f= \supp(f) \cup \{0\}$, which is a compact subset of $\bbR^n$.
Let $\what C_f$ denote its convex hull. 

Let $C \subset \bbR^n$ be a convex set containing $0$.
Define $C^* \subset (\bbR^n)^*$ to be the set of $\xi$ such that the hyperplane
$\brac{\xi,x}=1$ is disjoint from $C$. By convexity, 
\[ C^* = \{ \xi \mid \brac{\xi,x}<1 \text{ for all } x\in C \}. \]
Observe that $C^*$ is a convex set\footnote{$C^*$ is 
called \cite{Cassels} the dual (polar) set of $C$. 
Note that if $C$ is compact, then $C^*$ is open. If $0$ is an interior point of $C$, then $C^*$ is bounded.} containing $0$. 
If $\check C \subset (\bbR^n)^*$ is a convex set containing $0$, 
one similarly defines the dual $\check C^* \subset \bbR^n$.
Taking duals gives mutually inverse maps between the collection
of compact convex subsets of $\bbR^n$ 
containing $0$ and the collection of open convex subsets of $(\bbR^n)^*$ 
containing $0$. 

\begin{prop} \label{prop:convex}
The connected component of $(\bbR^n)^* \setminus \supp(\sM f)$ containing $0$ 
is equal to $(\what C_f)^*$. In particular, it is convex.  
\end{prop}

\begin{cor} \label{cor:notsurj}
The operator $\sM : \CCp \to \CCm$ is not surjective. 
\end{cor}

\begin{lem}\label{lem:halfplane}
Let $\xi_0 \in (\bbR^n)^* \setminus \{0\}$. If $\sM f$ vanishes on a 
neighborhood of the segment $[0,\xi_0] := \{t\xi_0 \mid 0\le t \le 1\}$, then 
$f$ vanishes on the half-space $\brac{\xi_0,x} \ge 1$. 
\end{lem}
\begin{proof} 
By replacing $f$ by a compactly supported function that is equal to 
$f$ outside of a small neighborhood of $0$, we may assume that
$f$ is compactly supported. 
There exists an open convex neighborhood $\check C$ of 
$[0,\xi_0]$ such that $\sM f$ vanishes on $\check C$.
Then $C=\check C^*$ is a compact convex subset of $\bbR^n$
and $C^* = \check C$, so the integral of $f$ along any hyperplane 
disjoint from $C$ vanishes. 
Therefore \cite[Corollary 2.8]{Helgason-Radon} implies
that $\supp f \subset C$. Since $\xi_0 \in \check C$, 
one sees that $C$ is contained in the half-space $\brac{\xi_0,x}< 1$.
\end{proof}

We have the support function $H:(\bbR^n)^* \to \bbR$ associated
to $C_f$, which is defined by
\[ H(\xi) = \sup\{ \brac{\xi , x} \mid x \in C_f \}. \]
For $\xi \ne 0$, the set $\{ x\mid \brac{\xi, x} = H(\xi)\}$ is a
supporting hyperplane of $\what C_f$. The function $H$ uniquely determines
the compact convex set $\what C_f$, and $(\what C_f)^* = H^{-1}(\bbR_{<1})$.

\begin{proof}[Proof of Proposition~\ref{prop:convex}]
It is clear that $H^{-1}(\bbR_{<1})$ is an open subset of 
$(\bbR^n)^* \setminus \supp(\sM f)$. 
Note that since $\supp(\sM f)$ is closed, Lemma~\ref{lem:halfplane}
implies that if $H(\xi)=1$ then $\xi \in \supp(\sM f)$. 
Thus $(\what C_f)^*=H^{-1}(\bbR_{<1})$ is also closed in $(\bbR^n)^* \setminus \supp(\sM f)$. 
\end{proof}

\section{$F$ complex} \label{sect:complex}
In this section we prove the invertibility of $M$ when $F=\bbC$.
Recall that in this case $K = U(n)$. 
The inversion formula is given in Theorem~\ref{thm:complex},
and a reformulation using the Mellin transform is given in Theorem~\ref{thm:convC}. The $K$-finiteness of $\C$ plays a crucial role in the proofs,
so we begin by recalling the classification of the $K$-isotypic components 
of $\C$.

\subsection{Spherical harmonics} 
Let $S^{2n-1}$ denote unit sphere of norm $1$ vectors in $\bbC^n = \bbR^{2n}$,
which has a natural action by $U(n)$.
Let $\eC(S^{2n-1})$ be the space of smooth $K$-finite functions on $S^{2n-1}$.
For nonnegative integers $p,q$, 
let $H^{p,q}$ denote the homogeneous polynomials of 
degree $p+q$ on $\bbR^{2n}$ that are harmonic and satisfy 
\[ Y(\lambda z_1,\dotsc,\lambda z_n) = \lambda^p \wbar\lambda^q Y(z_1,\dotsc,z_n)\] for $\lambda \in \bbC, (z_1,\dotsc,z_n)\in \bbC^n=\bbR^{2n}$. 

\begin{thm}[{\cite[Theorem 3.1]{JW}, \cite{Kostant}}] 
\label{thm:sphC}

Let $H^{p,q}|S^{2n-1}$ denote the space of harmonic polynomials 
restricted to $S^{2n-1}$. Then 

(i) $\C(S^{2n-1}) = \bigoplus_{p,q \ge 0} H^{p,q}|{S^{2n-1}}$ as $U(n)$-representations,

(ii) the $U(n)$-representations $H^{p,q}$ are irreducible and not isomorphic to each other.
\end{thm}

\subsection{Decomposing $\eC$ into $K$-isotypes} 
We have a decomposition $\bbC^n \setminus \{0\} = \bbR_{>0} \xt S^{2n-1}$, with $O(2n)$ (and hence $U(n)$) acting on the $S^{2n-1}$ component.
Let $\eC(\bbR_{>0})$ denote the space of smooth functions on $\bbR_{>0}$
and define the subspaces $\eC_\pm(\bbR_{>0}), \eC_c(\bbR_{>0})$
as in \S\ref{sect:C}.  

Theorem~\ref{thm:sphC} implies that there is a decomposition
\[ \C = \bigoplus_{p,q\ge 0} \C(\bbR_{>0}) \ot H^{p,q}. \]
For $u \in \C(\bbR_{>0})$ and $Y \in H^{p,q}$, we define $u\ot Y \in \C$ by
$(u \ot Y)(x) := u(\norm x) \cdot Y(\frac x{\norm x})$. 

\subsection{Radon inversion formula}

\subsubsection{} 
We have an isomorphism $\Inv:\Cm \to \Cp$ defined by 
\[ (\on{Inv} \varphi)(x) = \norm x^{-2n} \varphi\left(\frac{\wbar x}{\norm x^2}\right) \]
where $\wbar x$ is coordinate-wise conjugation.
Set $\wtilde M := \on{Inv}^{-1} \circ M$.  
Consider $\bbR_{>0}$ as a subgroup of diagonal matrices in $G$. 
Then it follows from \eqref{eqn:Mequiv} that $\wtilde M$ is a $K\xt \bbR_{>0}$ equivariant operator from $\Cp$ to $\Cp$. 

\subsubsection{} Let $A$ be the 
space of distributions on $\bbR_{>0}$ supported on $(0,1]$ (see \S\ref{subsect:A}). The action of $\bbR_{>0}$ on $\Cp$ induces an action of $A$.

\begin{thm} \label{thm:complex}
	The operator $M : \Cp \to \Cm$ is an isomorphism. For $\varphi \in \Cm(\bbR_{>0}) \ot H^{p,q}$,
the inverse $M^{-1} : \Cm \to \Cp$ is given by the formula
\[	M^{-1}\varphi = \beta_{p,q} * \Inv(\varphi) \]
where $\beta_{p,q}$ is the distribution on $\bbR_{>0}$ defined by 
\begin{equation}
\label{eqn:beta-complex}
\beta_{p,q}(t) = \frac{1}{2^{n+m-2} \pi^{n-1}\Gamma(m)} \prod_{j=1}^{n+m-1} \left( -\frac d{dt}\cdot t +p+q-2j \right)
\left(t^{-p-q-2n+1} (1-t^2)_+^{m-1} \right) dt 
\end{equation}
where $m=\min(p,q)$.
\end{thm}

\noindent
The derivative $\frac d {dt}$ is applied in the sense of generalized functions. 
The generalized function $(1-t^2)^\lambda_+$ is defined by analytic continuation for $\lambda \in \bbC$ (see the paragraph following the statement of Theorem~\ref{thm:real}).
In particular, the regularization of $\frac 2{\Gamma(m)} (1-t^2)_+^{m-1} dt$ at $m=0$ is equal to 
$\delta(1-t)$. 

\begin{cor} For any number $R$ one has $M^{-1}(\C_{\le -R})
\subset \C_{\ge R}$. \end{cor}
\begin{proof}
	Observe that $\beta_{p,q}$ is supported on $(0,1]$ for all $p,q$.
\end{proof}

\subsection{A formula for $\wtilde M$ in terms of convolution}
We consider the dot product on $S^{2n-1} \subset \bbC^n$ induced 
by the dot product on $\bbC^n$. 
For $t \in (-1,1)$, define $A_t : \C(S^{2n-1}) \to \C(S^{2n-1})$
such that $(A_t f)(x)$ is the average value of $f$ on the $(2n-3)$-sphere
$\{ \omega \in S^{2n-1} \mid \omega \cdot \wbar x = t \}$.
Then $A_t$ is $U(n)$-equivariant, so by Schur's lemma
it acts on $H^{p,q}|S^{2n-1}$ by a scalar $a_{p,q}(t)$. 
One observes that $a_{p,q}$ is a smooth function on $(-1,1)$, 
$\abs{a_{p,q}(t)} \le 1$ for all $t \in (-1,1)$, 
and $\lim_{t\to 1} a_{p,q}(t)=1$.

Suppose that $f \in \C(\bbR_{>0}) \ot H^{p,q}$ and there exist $C>0$ and $\sigma > n-1$
such that $\abs{f(x)} \le C \norm{x}^{-2\sigma}$ for all $x$ with $\norm x \ge 1$.  
One can deduce as in the real case that 
\begin{equation} \label{eqn:alphapq}
 \wtilde M f = \alpha_{p,q} * f 
\end{equation}
where $\alpha_{p,q}$ is the measure $\on{mes}(S^{2n-3}) \cdot t^{1-2n} \cdot a_{p,q}(t) (1-t^2)^{n-2}dt$
on the interval $(0,1)$ extended by zero to the whole $\bbR_{>0}$. 
The convolution $\alpha_{p,q} * f$ is well-defined due to the bound on $\abs{f(x)}$,
and $\on{mes}(S^{2n-3})$ denotes the surface area of the $(2n-3)$-sphere.
By considering zonal spherical functions, one can check \cite[Lemma 1.2]{Watanabe} that 
$a_{p,q}(t)$ is the scalar
multiple of the Jacobi polynomial $P_{\min(p,q)}^{(n-2,\abs{p-q})}(2t^2-1)$
normalized by $a_{p,q}(1)=1$.

The Mellin transform $\fM \alpha_{p,q}$ is defined for $s\in \bbC$ by integrating
$t^s$ against $\alpha_{p,q}$ if $\on{Re}(s) > 2n-2$.

\begin{thm} \label{thm:convC}
	The distribution $\alpha_{p,q}$ is invertible in $A$. The inverse $\beta_{p,q}$ is defined by
\eqref{eqn:beta-complex}. The Mellin transforms are given by 
\begin{equation}\label{eqn:mellinC}
	\fM \beta_{p,q}(s) = 
	\frac 1{\fM\alpha_{p,q}(s)} = \pi^{1-n} \frac{\Gamma(\frac{s+p+q}2)}{\Gamma(\frac{s+\abs{p-q}}2-n+1)} \cdot \frac{\Gamma(\frac{s-p-q}2-n+1)}{\Gamma(\frac{s-\abs{p-q}}2-n+1)}.
\end{equation}
\end{thm}
\noindent Theorem~\ref{thm:convC} implies Theorem~\ref{thm:complex}.

\subsection{Relation to the Fourier transform} 

Let $\eS(\bbC^n)$ denote the space of Schwartz functions on 
$\bbC^n$ and $\eS'(\bbC^n)$ the dual space of tempered distributions 
on $\bbC^n$. 
The Fourier transform is defined for an integrable function $f$ on $\bbC^n$ by 
\[ \eF f(\xi) = \int_{\bbC^n} f(x) e^{-2\pi i \on{Re}(\xi \cdot \wbar x)} dx. \]
This definition coincides with the one from \S\ref{sect:real-fourier}
by identifying $\bbC^n = \bbR^{2n}$. The Fourier transform can be
extended to an isomorphism of tempered distributions 
$\eF : \eS'(\bbC^n) \to \eS'(\bbC^n)$ .

Let $F : \eS'(\bbC) \to \eS'(\bbC)$ denote the Fourier transform
over $\bbC$. For $f \in \eS(\bbC^n)$, one gets $\eF f$ from the 
Radon transform by
\begin{equation}\label{eqn:F-Rf-complex}
	\eF f(r\omega) = F(\eR f(\wbar\omega,t))(r) 
\end{equation}
where $F$ is the Fourier transform with respect to the $t$ variable, 
$r\in \bbC$, and $\omega \in S^{2n-1}$ is 
a unit vector.
We have the following complex analog of Lemma~\ref{lem:Rdist}, which is 
proved in exactly the same way.
\begin{lem} \label{lem:Cdist}
Let $f$ be a locally integrable function
on $\bbC^n$ for which there exist $C>0$ and $\sigma >n-1$ such that 
$\abs{f(x)}\leq C \norm{x}^{-2\sigma}$ for all $x$ with $\norm x\ge 1$.
Then: 

(i) $\eR f$ is a locally integrable function on $S^{2n-1} \xt \bbC$.

(ii) $\eR f(\omega,t)$ is bounded for $\abs t \ge 1$.

(iii) The right hand side of \eqref{eqn:F-Rf-complex} is well-defined as a 
generalized function on $\bbC \xt S^{2n-1}$.

(iv) Equation \eqref{eqn:F-Rf-complex} holds as an equality between
generalized functions on $\bbR_{>0} \xt S^{2n-1}$. 
\end{lem}

\subsection{Proof of Theorem~\ref{thm:convC}} 
Let $Y \in H^{p,q}$ and define $f(x) = \norm{x}^{-s} \cdot Y(\frac x {\norm x})$
for $s\in \bbC$. If $2n-2 < \on{Re}(s) < 2n$, then $f$ is locally 
integrable on $\bbC^n$ and satisfies the hypothesis of Lemma~\ref{lem:Cdist}.
Moreover by \eqref{eqn:alphapq} and homogeneity of $f$
we see that 
\[	\eR f(\wbar \omega, t) = t^p \wbar t^q \abs{t}^{2n-2-p-q-s} \wtilde Mf(\omega)
	= t^p \wbar t^q \abs{t}^{2n-2-p-q-s} \fM \alpha_{p,q}(s) Y(\omega)
\]
as a locally integrable function on $S^{2n-1} \xt \bbC$.
Then Lemma~\ref{lem:Cdist} implies that 
\[ \eF f(r\omega) = F(t^p \wbar t^q \abs{t}^{2n-2-p-q-s})(r) 
\fM\alpha_{p,q}(s) Y(\omega) \]
as generalized functions on $\bbR_{>0} \xt S^{2n-1}$.

\begin{lem} If $2n-2 < \on{Re}(s) < 2n$, then 
\[ F(t^p \wbar t^q \abs{t}^{2n-2-p-q-s})(r) = 
	\pi^{s-2n+1} i^{-\abs{p-q}} \frac{\Gamma(\frac{2n+\abs{p-q}-s}2)}{\Gamma(\frac{s-2n+2+\abs{p-q}}2)}
		r^p \wbar r^q \abs r^{s-2n-p-q}  \]
as locally integrable functions on $\bbC$.
\end{lem}
\begin{proof}
	Apply Theorem~\ref{thm:SW} for $n=2$, $k = \abs{p-q}$, and 
	$Y(x_1,x_2) = (x_1 +i x_2)^{p-q}$ if $p\ge q$ or $Y(x_1,x_2)=(x_1-i x_2)^{q-p}$ if $p\le q$. 
\end{proof}

Alternatively, we can use Theorem~\ref{thm:SW} to find that 
$\eF f(x) = \gamma \norm{x}^{s-2n} Y(\frac x{\norm x})$, where 
$\gamma = i^{-p-q} \pi^{s-n} \Gamma(\frac{2n+p+q-s}2)/\Gamma(\frac{s+p+q}2)$. 
Comparing the two formulas we have derived for $\eF f$ and applying Euler's reflection formula, 
we conclude that 
\begin{equation}\label{eqn:Mapq}
	\fM \alpha_{p,q}(s) = \pi^{n-1} \frac{\Gamma(\frac{s+\abs{p-q}}2-n+1)\Gamma(\frac{s-\abs{p-q}}2 - n +1)}{\Gamma(\frac{s+p+q}2)\Gamma(\frac{s-p-q}2 -n+1)}, 
\end{equation} 
as stated in Theorem~\ref{thm:convC}. The equation holds \emph{a priori} for 
$2n-2 < \on{Re}(s) < 2n$, and we deduce by analytic continuation that it holds for all $s \in \bbC$, away from poles.

To finish the proof of Theorem~\ref{thm:convC}, it remains to show that
$(\fM\beta_{p,q})^{-1}$ is equal to the right hand side of \eqref{eqn:Mapq}. 
By considering the Beta function we see that
$\Gamma(\frac{s-p-q}2-n+1)/\Gamma(\frac{s-\abs{p-q}}2-n+1)$ is
the Mellin transform of $\nu(t)dt$, where 
\begin{equation}\label{eqn:nuc}
 \nu(t) = \frac 2{\Gamma(m)} t^{-p-q-2n+1} (1-t^2)_+^{m-1} 
\end{equation}
for $m=\min(p,q)$. Note that if $m=0$, then $\nu(t)dt = \delta(1-t)$.
Multiplying the right hand side of \eqref{eqn:nuc} by 
$\Gamma(\frac{s+p+q}2)/\Gamma(\frac{s+\abs{p-q}}2-n+1) = \prod_{j=1}^{n+m-1} (\frac{s+p+q}2-j)$ amounts to replacing $\nu$ by 
$L_{p,q} (\nu)$, where $L_{p,q}$ is the differential operator 
\[ 2^{1-n-m} \prod_{j=1}^{n+m-1} \left( -\frac d{dt}\cdot t+p+q-2j \right). \]  
Theorem~\ref{thm:convC} is proved.

In the case $n=2$, the formula \eqref{eqn:Mapq} is well-known (cf.~\cite[Lemma 7.23]{Wallach}, \cite[Proposition III.3.7]{Duflo}).

\end{document}